\documentclass[11pt]{article}

\usepackage{cite}
\usepackage{amsmath}
\usepackage{amsthm}
\usepackage{amscd}
\usepackage{amsbsy}
\usepackage{amssymb}
\usepackage{amsmath}
\usepackage{amsfonts}
\usepackage{amsopn}
\usepackage{a4wide}

\theoremstyle{plain}
\newtheorem{theorem}{Theorem}[section]

\newtheorem{proposition}[theorem]{Proposition}

\theoremstyle{definition}
\newtheorem{definition}{Definition}[section]
\newtheorem{example}[theorem]{Example}

\theoremstyle{remark}
\newtheorem{remark}{Remark}[section]

\newcommand{\N}{\mathbb{N}}

\newcommand{\R}{\mathbb{R}}

\newcommand{\T}{\mathbb{T}}

\newcommand{\Z}{\mathbb{Z}}

\begin{document}

\title{Fractional Derivatives and Integrals on Time Scales\\
via the Inverse Generalized Laplace Transform\thanks{Submitted 20-Aug-2010; 
accepted 11-Nov-2010. Ref: Int. J. Math. Comput. 11 (2011), no.~J11, 1--9.}}

\author{\textbf{Nuno R. O. Bastos$^1$, Dorota Mozyrska$^2$ and Delfim F. M. Torres$^3$}}

\date{$^1$Department of Mathematics, School of Technology\\
Polytechnic Institute of Viseu\\
3504-510 Viseu, Portugal\\
nbastos@mat.estv.ipv.pt\\[0.3cm]
$^2$Faculty of Computer Science\\
Bia{\l}ystok  University of Technology\\
15-351 Bia{\l}ystok, Poland\\
d.mozyrska@pb.edu.pl\\[0.3cm]
$^3$Department of Mathematics\\
University of Aveiro\\
3810-193 Aveiro, Portugal\\
delfim@ua.pt}

\maketitle


\begin{abstract}
\noindent \emph{We introduce a fractional calculus on time scales
using the theory of delta (or nabla) dynamic equations.
The basic notions of fractional order integral
and fractional order derivative
on an arbitrary time scale are proposed, using
the inverse Laplace transform on time scales.
Useful properties of the new fractional operators are proved.}

\medskip

\noindent\textbf{Keywords:} fractional derivatives and integrals,
time scales, Laplace transform.

\medskip

\noindent\textbf{2010 Mathematics Subject Classification:} 26A33, 26E70, 44A10.
\end{abstract}


\section{Introduction}

The fractional calculus deals with extensions of derivatives and integrals
to noninteger orders.
It represents a powerful tool in applied
mathematics to study a myriad of problems from different fields
of science and engineering, with many break-through results found in
mathematical physics, finance, hydrology, biophysics,
thermodynamics, control theory, statistical mechanics,
astrophysics, cosmology and bioengineering
\cite{kilbas,oldman,podlubny,samko}.

The fractional calculus may be approached via the theory
of linear differential equations. Analogously, starting
with a linear difference equation, we are led to a
definition of fractional difference of an arbitrary
order \cite{comNunoRui:confOrtigueira09}.
Our main objective is to introduce here a fractional calculus
on an arbitrary time scale, using the theory
of delta (or nabla) differential equations.

The analysis on time scales is a fairly new area of research.
It was introduced in 1988 by Stefan Hilger
and his Ph.D. supervisor Bernd Aulbach \cite{Aul:Hil,Hilger}.
It combines the traditional areas of continuous and discrete analysis
into one theory. After the publication of three textbooks in this area
\cite{Bh,Bh:Pet:03,book:Lak},
more and more researchers are getting involved
in this fast-growing field.

Recently, two attempts have been made to provide
a general definition of fractional derivative on
an arbitrary time scale \cite{Anastassiou,emptyPaper}.
These two works address a very interesting question,
but unfortunately there is a small inconsistency
in the very beginning of both studies. Indeed,
investigations of \cite{Anastassiou,emptyPaper} are based
on the following definition of generalized polynomials on time scales
$h_\alpha : \mathbb{T} \times \mathbb{T}\rightarrow \mathbb{R}$:
\begin{equation}
\label{e:d}
\begin{gathered}
h_0(t,s) = 1 ,\\
h_{\alpha + 1}(t,s) = \int_s^t h_\alpha(\tau,s)\Delta\tau.
\end{gathered}
\end{equation}
Recursion \eqref{e:d} provides a definition only
in the case $\alpha \in \mathbb{N}_0$,
and there is no hope to define polynomials
$h_\alpha$ for real or complex indices $\alpha$ with \eqref{e:d}.
Here we propose a different approach to the subject
based on the Laplace transform \cite{laplace}.

The paper is organized as follows.
In Sect.~\ref{sec:prl} we briefly review the
basic notions of Riemann--Liouville and Caputo
fractional integration and differentiation
(Sect.~\ref{prl:frac}) as well as necessary tools
from time scales (Sect.~\ref{prl:ts}).
Our results are then given in Sect.~\ref{sec:mr}:
we introduce the concept of fractional integral and fractional
derivative on an arbitrary time scale $\T$ (Sect.~\ref{mr:frac});
we then prove some important properties of the fractional integrals
and derivatives (Sect.~\ref{mr:prop}).


\section{Preliminaries}
\label{sec:prl}

The following definitions and basic results
serve essentially to fix notations.
The reader is assumed to be familiar with both fractional
calculus and calculus on time scales. For an introduction
to fractional and time scale theories we refer to
the books \cite{kilbas,oldman,podlubny,samko}
and \cite{Bh,Bh:Pet:03,book:Lak}.


\subsection{Caputo and Riemann--Liouville fractional derivatives}
\label{prl:frac}

Let $f$ be an arbitrary integrable function. By $\mathcal{D}^{-\alpha}_{a}f$
we denote the fractional integral of $f$ of order $\alpha\in(0,1)$ on $[0,t]$,
defined as
\begin{equation*}
\left(\mathcal{D}^{-\alpha}_{a}f\right)(t)
=\frac{1}{\Gamma(\alpha)}\int_0^t\frac{f(\tau)}{(t-\tau)^{1-\alpha}}d\tau
\end{equation*}
with $\Gamma$ the well-known gamma function.
For an arbitrary real number $\alpha$, the Riemann--Liouville
and Caputo fractional derivatives are defined, respectively,
as in \cite{ChenPodlubny}:
\begin{equation*}
\left(D^{\alpha}_{a}f\right)(t)
=\left(\frac{d^{[\alpha]+1}}{dt^{[\alpha]+1}}\left(\mathcal{D}_a^{-\left([\alpha]
-\alpha+1\right)}f\right)\right)(t)
\end{equation*}
and
\begin{equation*}
\left(\mbox{}^CD^{\alpha}_{a}f\right)(t)
=\left(\mathcal{D}_a^{-\left([\alpha]-\alpha
+1\right)}\left(\frac{d^{[\alpha]
+1}}{dt^{[\alpha]+1}}f\right)\right)(t)\,,
\end{equation*}
where $[\alpha]$ is the integer part of $\alpha$.

The notion of Riemann--Liouville derivative
is the preferred fractional derivative among mathematicians,
while Caputo fractional derivative is the preferred one among engineers.
If $f(a)=f'(a)=\cdots=f^{(n-1)}(a)=0$, then
both Riemann--Liouville and Caputo derivatives coincide.
In particular, for $\alpha\in (0,1)$ and $f(a)=0$ one has
$\mbox{}^CD^{\alpha}_{a}f(t)=D^{\alpha}_{a}f(t)$.

Next proposition gives the Laplace transform
of the Caputo fractional derivative.

\begin{proposition}[\cite{kilbas}]
\label{prop:Kilbas}
Let $\alpha>0$, $n$ be the integer such that $n-1<\alpha\leq n$,
and $f$ a function satisfying $f \in C^n(\R^+)$, $f^{(n)}\in L_1(0,t_1)$, $t_1>0$,
and $|f^{(n)}(t)|\leq B \mathrm{e}^{q_0 t}$, $t>t_1>0$.
If the Laplace transforms $\mathcal{L}[f](z)$
and $\mathcal{L}[D^n f](z)$ exist,
and $\lim\limits_{t\rightarrow+\infty} D^k f(t)=0$ for $k=0,\ldots,n-1$, then
\begin{equation*}
\mathcal{L}\left[\mbox{}^CD^{\alpha}_{0}f\right](z)
=z^{\alpha}\mathcal{L}\left[f\right](z)
-\sum_{k=0}^{n-1}f^{(k)}(0)z^{\alpha-k-1}\,.
\end{equation*}
\end{proposition}

\begin{remark}
If $\alpha\in(0,1]$, then
$\mathcal{L}\left[\mbox{}^CD^{\alpha}_{0}f\right](z)
=z^{\alpha}\mathcal{L}\left[f\right](z)-f(0)z^{\alpha-1}$.
\end{remark}


\subsection{The Laplace transform on time scales}
\label{prl:ts}

Throughout the paper, $\T$ is an arbitrary time scale with bounded graininess,
\textrm{i.e.}, $0<\mu_{min}\leq \mu(t)\leq \mu_{max}$ for all $ t \in \T$.
Let $t_0 \in \T$ be fixed. We define functions
$h_k(\cdot, t_0):\T \rightarrow\R$, $k\in\N_0$, recursively as in \eqref{e:d}:
$h_0(t,t_0)\equiv 1$, $h_{k+1}(t,t_0)=\int_{t_0}^th_k(\tau,t_0)\Delta \tau$.
Functions $h_k(\cdot, t_0)$ are called \emph{generalized polynomials}
on the time scale $\T$. By $h^{\Delta}_k(t, t_0)$
we denote the delta derivative of $h_k(t, t_0)$ with respect to $t$.
For $t\in\T^{\kappa}$ one has $h^{\Delta}_0(t, t_0)=0$
and $h^{\Delta}_k(t,t_0)=h_{k-1}(t,t_0)$, $k\in\N$.

\begin{example}
In \cite{Bh} one can find exact formulas of generalized polynomials
for some particular time scales $\T$:
if $\T=\R$, then $h_k(t,t_0)=\frac{(t-t_0)^k}{k!}$;
if $\T=\Z$, then $h_k(t,t_0)=\frac{(t-t_0)^{(k)}}{k!}=\binom{t-t_0}{k}$,
where $(t-t_0)^{(0)}=1$ and $(t-t_0)^{(k)}=(t-t_0)(t-t_0-1)\cdots(t-t_0-k+1)$, $k\in\N$.
\end{example}

The Laplace transform on time scales is studied in \cite{laplace}.
\begin{definition}[\cite{laplace}]
\label{def:ltts}
For $f:\T\rightarrow \R$, the generalized Laplace transform of $f$,
denoted by $\mathcal{L}_{\T}[f]$, is defined by
\begin{equation*}
\mathcal{L}_{\T}[f](z)=F(z)
:= \int_0^{\infty} f(t)\mathrm{e}_{\ominus z}(\sigma(t),0)\Delta t\,.
\end{equation*}
\end{definition}

\begin{remark}
In view of Definition~\ref{def:ltts},
the Laplace transform $\mathcal{L}$
of Proposition~\ref{prop:Kilbas}
can be written as $\mathcal{L}_\R$.
\end{remark}

\begin{definition}[\cite{laplace}]
The function $f:\T\rightarrow \R$ is said to be of \emph{exponential type I}
if there exist constants $M, c>0$ such that $|f(t)|\leq M \textrm{e}^{c t}$.
Furthermore, $f$ is said to be of \emph{exponential type II}
if there exist constants $M, c>0$ such that $|f(t)|\leq M \mathrm{e}_c(t,0)$.
\end{definition}

The time scale exponential function $\mathrm{e}_c(t,0)$ is of type II
while generalized polynomials $h_k(t,0)$ are of type I.

\begin{theorem}[\cite{laplace}]
If $f$ is of exponential type II with exponential constant $c$, then
the delta integral $\int_0^{\infty} f(t) \mathrm{e}_{\ominus}(\sigma(t),0)\Delta t$
converges absolutely for $z\in D$.
\end{theorem}

From \cite[Theorem 1.3]{laplace} and mathematical induction
one obtains the following result:

\begin{proposition}
Let $F$ be the generalized Laplace transform of $f:\T\rightarrow \R$. Then,
\begin{equation}
\label{lap:derivative}
\mathcal{L}_{\T}[f^{\Delta^n}](z)=z^nF(z)-\sum_{k=0}^{n-1}z^{n-k-1}f^{\Delta^{k}}(0)\,.
\end{equation}
\end{proposition}

\begin{remark}
If $F$ is the Laplace transform of $f:\T\rightarrow \R$, then
the Laplace transform of the delta derivative of $f$ is given by
\[
\mathcal{L}_{\T}[f^{\Delta}](z)=zF(z)-f(0)\,.
\]
\end{remark}

Our work is motivated by the following result:

\begin{theorem}[Inversion formula of the Laplace transform \cite{laplace}].
Suppose that $F$ is analytic in the region
$R \mathrm{e}_{\mu}(z)>R \mathrm{e}_{\mu}(c)$
and $F(z)\rightarrow 0$ uniformly as $|z|\rightarrow\infty$ in this region.
Assume $F$ has finitely many regressive poles of finite order
$\{z_1,z_2,\ldots, z_n\}$ and $\tilde{F}_{\R}(z)$ is the transform
of the function $\tilde{f}(t)$ on $\R$ that corresponds to the transform
$F(z)=F_{\T}(z)$ of $f(t)$ on $\T$. If
\[
\int_{c-i\infty}^{c+i\infty}\left|\tilde{F}_{\R}(z)\right||dz|<\infty\,,
\]
then
\begin{equation*}
f(t)=\sum_{i=1}^n\mbox{Res}_{z=z_i} \mathrm{e}_z(t,0)F(z)
\end{equation*}
has transform $F(z)$ for all $z$ with $Re(z)>c$.
\end{theorem}

\begin{remark}
The inverse transform of $F(z)=\frac{1}{z^{k+1}}$, $k\in\N$, is $h_k(\cdot,0)$.
Moreover, the inversion formula gives the claimed inverses for any of the
elementary functions that were presented in the table of Laplace transform
in \cite{lap_Bohner}.
\end{remark}

Using the series $\mathrm{e}_z(t,0)=\sum_{k=0}^{+\infty}z^kh_k(t,0)$,
we see that generalized polynomials $h_k(t,0)$ give us the
difference between time scales for the inverse Laplace images.


\section{Main Results}
\label{sec:mr}

We begin by introducing the definition of fractional integral
and fractional derivative on an arbitrary time scale $\T$.


\subsection{Fractional derivative and integral on time scales}
\label{mr:frac}

Similarly to the classical case,
the Laplace transform of a Caputo fractional
derivative of order $\alpha\in (0,1]$ is given by
$\mathcal{L}\left[\mbox{}^CD^{\alpha}_{0^+}f\right](z)
=z^{\alpha}\mathcal{L}\left[f\right](z)
-f(0^+)z^{\alpha-1}$. Simultaneously, the generalized
Laplace transform on time scales gives unification and
extension of the classical results. Important to us,
the Laplace transform of the $\Delta$-derivative is given by the formula
$\mathcal{L}_{\T}[f^{\Delta}](z)=zF(z)-f(0)$. Our idea is to
define the fractional derivative on time scales via the
inverse Laplace transform formula for the complex function
$G(z)=z^{\alpha}\mathcal{L}_{\T}[f](z)-f(0^+)z^{\alpha-1}$.
Furthermore, for $\alpha\in(n-1,n]$, $n\in\N$,
we use a generalization of \eqref{lap:derivative}
to define fractional derivatives on times
scales for higher orders $\alpha$.

\begin{definition}[Fractional integral on time scales]
Let $\alpha>0$, $\T$ be a time scale,
and $f : \mathbb{T} \rightarrow \mathbb{R}$.
The fractional integral of $f$ of order $\alpha$
on the time scale $\T$, denoted by $I_{\T}^{\alpha}f$,
is defined by
\[
I_{\T}^{\alpha}f(t)
=\mathcal{L}_{\T}^{-1}\left[\frac{F(z)}{z^{\alpha}}\right](t)\,.
\]
\end{definition}

\begin{definition}[Fractional derivative on time scales]
\label{def:md}
Let $\T$ be a time scale,
$F(z)=\mathcal{L}_{\T}[f](z)$, and $\alpha\in (n-1,n]$, $n\in\N$.
The fractional derivative of function $f$ of order $\alpha$
on the time scale $\T$, denoted by $f^{(\alpha)}$, is defined by
\begin{equation}
\label{frac:def:n}
f^{(\alpha)}(t)=\mathcal{L}^{-1}_{\T}\left[z^{\alpha}F(z)
-\sum_{k=0}^{n-1}f^{\Delta^k}(0^+)z^{\alpha-k-1}\right](t)\,.
\end{equation}
\end{definition}

\begin{remark}
For $\alpha\in (0,1]$ we have
\begin{equation*}
f^{(\alpha)}(t)=\mathcal{L}^{-1}_{\T}\left[z^{\alpha}F(z)
-f(0^+)z^{\alpha-1}\right](t)\,.
\end{equation*}
\end{remark}


\subsection{Properties}
\label{mr:prop}

We begin with two trivial but important remarks about
the fractional integral and the fractional derivative
operators just introduced.

\begin{remark}
As the inverse Laplace transform is linear,
we also have linearity for the new fractional integral and derivative:
\begin{equation*}
\begin{split}
I_{\T}^{\alpha}(af+bg)(t)
&=aI_{\T}^{\alpha}f(t)+bI_{\T}^{\alpha}g(t),\\
(af+bg)^{(\alpha)}(t)
&=a f^{(\alpha)}(t) + b g^{(\alpha)}(t)\,.
\end{split}
\end{equation*}
\end{remark}

\begin{remark}
Since $h_0(t) \equiv 1$ for any time scale $\T$, from the
definition of Laplace transform and the fractional derivative
we conclude that $h_0^{(\alpha)}(t)=0$. For $\alpha>1$
one has $h_1^{(\alpha)}(t)=0$.
\end{remark}

We now prove several important properties
of the fractional integrals and fractional derivatives
on arbitrary time scales.

\begin{proposition}
Let $\alpha\in(n-1,n]$, $n \in \N$. If $k\leq n-1$, then
\[
h_k^{(\alpha)}(t,0)=0\,.
\]
\end{proposition}

\begin{proof}
From \eqref{frac:def:n} it follows that
\begin{equation*}
\begin{split}
h_k^{(\alpha)}(t,0)&=\mathcal{L}^{-1}_{\T}\left[\frac{z^{\alpha}}{z^{k+1}}
-\sum_{i=0}^{n-1}h_k^{\Delta^i}(0)z^{\alpha-i-1}\right](t)
=\mathcal{L}^{-1}_{\T}\left[\frac{z^{\alpha}}{z^{k+1}}
-\sum_{i=0}^{k}h_{k-i}(0)z^{\alpha-i-1}\right](t)\\
&=\mathcal{L}^{-1}_{\T}\left[\frac{z^{\alpha}}{z^{k+1}}-z^{\alpha-k-1}\right](t)=0\,.
\end{split}
\end{equation*}
\end{proof}

\begin{proposition}
Let $\alpha\in(n-1,n]$, $n \in \N$. If $k\geq n$, then
\[
h_k^{(\alpha)}(t,0)
=\mathcal{L}^{-1}_{\T}\left[\frac{1}{z^{k+1-\alpha}}\right](t) \,.
\]
\end{proposition}

\begin{proof}
From \eqref{frac:def:n} we have
\begin{equation*}
\begin{split}
h_k^{(\alpha)}(t)&=\mathcal{L}^{-1}_{\T}\left[\frac{z^{\alpha}}{z^{k+1}}
-\sum_{i=0}^{n-1}h_k^{\Delta^i}(0)z^{\alpha-i-1}\right](t)
=\mathcal{L}^{-1}_{\T}\left[\frac{z^{\alpha}}{z^{k+1}}
-\sum_{i=0}^{n-1}h_{k-i}(0)z^{\alpha-i-1}\right](t)\\
&=\mathcal{L}^{-1}_{\T}\left[\frac{z^{\alpha}}{z^{k+1}}\right](t)
=\mathcal{L}^{-1}_{\T}\left[\frac{1}{z^{k+1-\alpha}}\right](t)\,.
\end{split}
\end{equation*}
\end{proof}

\begin{proposition}
Let $\alpha\in(n-1,n]$, $n \in \N$.
If $c(t) \equiv m$, $m\in \mathbb{R}$, then
\[
c^{(\alpha)}(t)=0.
\]
\end{proposition}

\begin{proof}
From the linearity of the inverse Laplace transform
and the fact that $h^{(\alpha)}_0(t,0)=0$, it follows that
$c^{(\alpha)}(t)=(m\times 1)^{(\alpha)}=(m\times
h_0(t,0))^{(\alpha)}=m\times h^{(\alpha)}_0(t,0)=m\times 0 = 0$.
\end{proof}

\begin{proposition}
\label{prop:comp_integrals}
Let $\alpha,\beta>0$. Then,
\begin{equation*}
I_{\T}^{\beta}\left(I_{\T}^{\alpha}f\right)(t)
=I_{\T}^{\alpha+\beta}f(t)\,.
\end{equation*}
\end{proposition}

\begin{proof}
By definition we have
\[
I_{\T}^{\beta}\left(I_{\T}^{\alpha}f\right)(t)
=\mathcal{L}_{\T}^{-1}\left[z^{-\beta}\mathcal{L}_{\T}\left[I_{\T}^{\alpha}f\right]\right](t)
=\mathcal{L}_{\T}^{-1}\left[s^{-\alpha-\beta}F(z)\right](t)
=I_{\T}^{\alpha+\beta}f(t)\,.
\]
\end{proof}

\begin{proposition}
\label{prop:comp}
Let $\alpha,\beta\in(0,1]$.
\begin{itemize}
\item[(a)] If $\alpha+\beta\leq 1$, then
\[
\left(f^{(\alpha)}\right)^{(\beta)}(t)
=f^{(\alpha+\beta)}(t)-\mathcal{L}^{-1}_{\T}\left[z^{\beta-1}f^{(\alpha)}(0)\right](t)\,.
\]
\item[(b)] If $1<\alpha+\beta\leq 2$, then
\[
\left(f^{(\alpha)}\right)^{(\beta)}(t)=f^{(\alpha+\beta)}(t)
-\mathcal{L}^{-1}_{\T}\left[z^{\beta-1}f^{(\alpha)}(0)\right](t)+
\mathcal{L}^{-1}_{\T}\left[z^{\alpha+\beta-2}f^{\Delta}(0)\right](t)\,.
\]
\item[(c)] If $\beta\in(0,1]$, then
\[
\left(f^{\Delta}\right)^{(\beta)}(t)=f^{(\beta+1)}(t)\,.
\]
\end{itemize}
\end{proposition}

\begin{proof}
(a) Let $\alpha+\beta\leq 1$. Then
\begin{equation*}
\begin{split}
\left(f^{(\alpha)}\right)^{(\beta)}(t)
&=\mathcal{L}^{-1}_{\T}\left[z^{\beta}\mathcal{L}\left[f^{(\alpha)}\right](z)
-z^{\beta-1}f^{(\alpha)}(0)\right](t)\\
&=\mathcal{L}^{-1}_{\T}\left[z^{\alpha+\beta}F(z)-z^{\alpha+\beta-1}f(0)\right](t)
- \mathcal{L}^{-1}_{\T}\left[z^{\beta-1}f^{(\alpha)}(0)\right](t)\\
&=f^{(\alpha+\beta)}(t)
-\mathcal{L}^{-1}_{\T}\left[z^{\beta-1}f^{(\alpha)}(0)\right](t)\,.
\end{split}
\end{equation*}
(b) For $1<\alpha+\beta\leq 2$ we have
\[f^{(\alpha+\beta)}(t)=
\mathcal{L}^{-1}_{\T}\left[z^{\alpha+\beta}\mathcal{L}\left[f\right](z)
-z^{\alpha+\beta-1}f(0)-z^{\alpha+\beta-2}f^{\Delta}(0)\right](t)\,.\]
Hence,
\[\left(f^{(\alpha)}\right)^{(\beta)}(t)=f^{(\alpha+\beta)}(t)
+\mathcal{L}^{-1}_{\T}\left[z^{\alpha+\beta-2}f^{\Delta}(0)\right](t)
-\mathcal{L}^{-1}_{\T}\left[z^{\beta-1}f^{(\alpha)}(0)\right](t)\,. \]
(c) The intended relation follows from (b) by choosing $\alpha=1$.
\end{proof}

In general $\left(f^{(\alpha)}\right)^{(\beta)}(t)
\neq\left(f^{(\beta)}\right)^{(\alpha)}(t)$.
However, the equality holds in particular cases:

\begin{proposition}
If $\alpha+\beta\leq 1$ and $f(0)=0$, then
$\left(f^{(\alpha)}\right)^{(\beta)}(t)
=\left(f^{(\beta)}\right)^{(\alpha)}(t)$.
\end{proposition}

\begin{proof}
It follows from item (a) of Proposition~\ref{prop:comp}.
\end{proof}

\begin{proposition}
Let $\alpha\in (n-1,n]$, $n \in \N$, and $\lim\limits_{t\rightarrow
0^+}f^{\Delta^k}(t)=f^{\Delta^k}(0^+)$ exist, $k=0,\ldots,n-1$.
The following equality holds:
\begin{equation*}
I^{\alpha}_{\T}\left(f^{(\alpha)}\right)(t)
=f(t)-\sum_{k=0}^{n-1}f^{\Delta^k}(0^+)h_k(t)\,.
\end{equation*}
\end{proposition}

\begin{proof}
Let $F(z)=\mathcal{L}_{\T}\left[f\right](z)$. Then,
\begin{equation*}
\begin{split}
I^{\alpha}_{\T}\left(f^{(\alpha)}\right)(t)
&=\mathcal{L}_{\T}^{-1}\left[z^{-\alpha}\mathcal{L}_{\T}\left[f^{(\alpha)}\right](z)\right](t)
=\mathcal{L}_{\T}^{-1}\left[F(z)-\sum_{k=0}^{n-1}f^{\Delta^k}(0^+)\frac{z^{\alpha-k-1}}{z^{\alpha}}\right](t)\\
&=f(t)-\sum_{k=0}^{n-1}f^{\Delta^k}(0^+)\mathcal{L}_{\T}^{-1}\left[\frac{1}{z^{k+1}}\right](t)
=f(t)-\sum_{k=0}^{n-1}f^{\Delta^k}(0^+)h_k(t)\,.
\end{split}
\end{equation*}
\end{proof}

\begin{proposition}
Let $\alpha\in (n-1,n]$, $n \in \N$, and $F(z)=\mathcal{L}_{\T}\left[f\right](z)$.
If $\lim\limits_{z\rightarrow\infty}F(z)=0$ and
$\lim\limits_{z\rightarrow\infty}\frac{F(z)}{z^{\alpha-k}}=0$,
$k \in \{1,\ldots,n\}$, then
$\left(I^\alpha_\mathbb{T}f\right)^{(\alpha)}(t)=f(t)$.
\end{proposition}

\begin{proof}
Firstly let us notice that
$\lim\limits_{t\rightarrow 0^+}\left(I_{\T}^{\alpha}f\right)^{\Delta^k}(t)=0$
for $k\in\{0,\ldots, n-1\}$. For that we check that
$\lim\limits_{z\rightarrow \infty}z\mathcal{L}\left[I^{\alpha}_{\T}f\right](z)
=\lim\limits_{z\rightarrow \infty}z\frac{F(z)}{z^{\alpha}}=0$.
Nextly let us assume that it holds for $i=0,\ldots,k-1$ and let us calculate
\begin{equation*}
\begin{split}
\lim\limits_{z\rightarrow \infty}z\mathcal{L}\left[\left(I^{\alpha}_{\T}
f\right)^{\Delta^k}\right](z)
&=\lim\limits_{z\rightarrow \infty}z\left(z^k\frac{F(z)}{z^{\alpha}}
-\sum_{i=0}^{k-1}z^{k-1-i}\left(I_{\T}^{\alpha}\right)^{\Delta^i}(0)\right)\\
&= \lim\limits_{z\rightarrow \infty}z\left(z^k\frac{F(z)}{z^{\alpha}}\right)
=\lim\limits_{z\rightarrow \infty}\frac{F(z)}{z^{\alpha-k-1}}=0 \,.
\end{split}
\end{equation*}
Then we easily conclude that
\[
\left(I^\alpha_\mathbb{T}f\right)^{(\alpha)}(t)
=\mathcal{L}_{\T}^{-1}\left[z^{\alpha}\frac{F(z)}{z^{\alpha}}\right](t)=f(t)\,.
\]
\end{proof}

The convolution of two functions
$f:\T\rightarrow\R$ and $g:\T\times \T\rightarrow\R$
on time scales, where $g$ is rd-continuous with respect to the first
variable, is defined in \cite{BL,laplace}:
\begin{equation*}
\left(f\ast g \right)(t)=\int_0^t f(\tau)g(t,\sigma(\tau))\Delta \tau\,.
\end{equation*}
As function $g$ we can consider, \textrm{e.g.},
$\mathrm{e}_c(t,t_0)$ or $h_k(t,t_0)$.

\begin{proposition}
Let $t_0\in \T$. If $\alpha\in(0,1)$, then
\begin{equation*}
\left(f\ast g(\cdot,t_0)\right)^{(\alpha)}(t)
=\left(f^{(\frac{\alpha}{2})}\ast g^{(\frac{\alpha}{2})}(\cdot,t_0)\right)(t)
=\left(f^{(\alpha)}\ast g(\cdot,t_0)\right)(t)\,,
\end{equation*}
where we assume the existence of the involved derivatives.
\end{proposition}

\begin{proof}
From the convolution theorem for the generalized Laplace transform
\cite[Theorem 2.1]{laplace},
\[
\mathcal{L}_{\T}\left[\left(f\ast g(\cdot,t_0)\right)^{(\alpha)}\right](z)
=z^{\alpha}F(z)G(z)\,.
\]
Hence,
\begin{equation*}
\begin{split}
\left(f\ast g(\cdot,t_0)\right)^{(\alpha)}(t)
&=\mathcal{L}^{-1}_{\T}\left[z^{\alpha}F(z)G(z)\right](t)
=\mathcal{L}^{-1}_{\T}\left[z^{\frac{\alpha}{2}}F(z)\right](t)
\mathcal{L}^{-1}_{\T}\left[z^{\frac{\alpha}{2}}G(z)\right](t)\\
&=\left(f^{(\alpha/2)}\ast g^{(\alpha/2)}(\cdot,t_0)\right)(t)\,.
\end{split}
\end{equation*}
Equivalently, we can write
\[
\mathcal{L}^{-1}_{\T}\left[z^{\alpha}F(z)G(z)\right](t)
=\mathcal{L}^{-1}_{\T}\left[z^{\alpha}F(z)\right](t)
\mathcal{L}^{-1}_{\T}\left[G(z)\right](t)=\left(f^{(\alpha)}\ast g(\cdot,t_0)\right)(t)\,.
\]
\end{proof}


\section*{Acknowledgements}

This work is part of the first author's
Ph.D. project carried out at the University of Aveiro.
The financial support of the Polytechnic Institute of Viseu and
\emph{The Portuguese Foundation for Science and Technology} (FCT),
through the ``Programa de apoio \`{a} forma\c{c}\~{a}o avan\c{c}ada
de docentes do Ensino Superior Polit\'{e}cnico'',
Ph.D. fellowship SFRH/PROTEC/49730/2009, is here gratefully acknowledged.
The second author was supported by BUT grant
S/WI/1/08; the third author by the R\&D unit CIDMA via FCT.




\begin{thebibliography}{99}

\bibitem{Anastassiou}
G. A. Anastassiou,
Principles of delta fractional calculus on time scales and inequalities,
Math. Comput. Modelling {\bf 52} (2010), no.~3-4, 556--566.

\bibitem{Aul:Hil}
B. Aulbach\ and\ S. Hilger,
A unified approach to continuous and discrete dynamics,
in {\it Qualitative theory of differential equations (Szeged, 1988)},
37--56, Colloq. Math. Soc. J\'anos Bolyai, 53, North-Holland, Amsterdam, 1990.

\bibitem{comNunoRui:confOrtigueira09}
N. R. O. Bastos, R. A. C. Ferreira\ and\ D. F. M. Torres,
Discrete-time fractional variational problems,
Signal Process. {\bf 91} (2011), no.~3, 513--524.
{\tt arXiv:1005.0252}

\bibitem{BL}
M. Bohner\ and\ D. A. Lutz,
Asymptotic expansions and analytic dynamic equations,
ZAMM Z. Angew. Math. Mech. {\bf 86} (2006), no.~1, 37--45.

\bibitem{Bh}
M. Bohner\ and\ A. Peterson,
{\it Dynamic equations on time scales},
Birkh\"auser Boston, Boston, MA, 2001.

\bibitem{lap_Bohner}
M. Bohner\ and\ A. Peterson,
Laplace transform and $Z$-transform: unification and extension,
Methods Appl. Anal. {\bf 9} (2002), no.~1, 151--157.

\bibitem{Bh:Pet:03}
M. Bohner\ and\ A. Peterson,
{\it Advances in dynamic equations on time scales},
Birkh\"auser Boston, Boston, MA, 2003.

\bibitem{laplace}
J. M. Davis, I. A. Gravagne, B. J. Jackson, R. J. M. Marks\ and \ A. A. Ramos,
The Laplace transform on time scales revisited,
J. Math. Anal. Appl. {\bf 332} (2007), no.~2, 1291--1307.

\bibitem{Hilger}
S. Hilger, Analysis on measure chains---a unified approach
to continuous and discrete calculus,
Results Math. {\bf 18} (1990), no.~1-2, 18--56.

\bibitem{kilbas}
A. A. Kilbas, H. M. Srivastava\ and\ J. J. Trujillo,
{\it Theory and applications of fractional differential equations},
Elsevier, Amsterdam, 2006.

\bibitem{book:Lak}
V. Lakshmikantham, S. Sivasundaram\ and\ B. Kaymakcalan,
{\it Dynamic systems on measure chains}, Kluwer Acad. Publ., Dordrecht, 1996.

\bibitem{ChenPodlubny}
Y. Li, Y. Q. Chen\ and\ I. Podlubny,
Mittag-Leffler stability of fractional order nonlinear
dynamic systems, Automatica {\bf 45} (2009), no.~8, 1965--1969.

\bibitem{oldman}
K. B. Oldham\ and\ J. Spanier,
{\it The fractional calculus}, Academic Press
[A subsidiary of Harcourt Brace Jovanovich, Publishers], New York, 1974.

\bibitem{podlubny}
I. Podlubny, {\it Fractional differential equations},
Academic Press, San Diego, CA, 1999.

\bibitem{emptyPaper}
M. R. S. Rahmat\ and\ M. S. Md. Noorani,
Fractional integrals and derivatives on time scales with an application,
Comput. Math. Appl. (2009), DOI: 10.1016/j.camwa.2009.10.013

\bibitem{samko}
S. G. Samko, A. A. Kilbas\ and\ O. I. Marichev,
{\it Fractional integrals and derivatives},
Translated from the 1987 Russian original, Gordon and Breach, Yverdon, 1993.

\end{thebibliography}
\end{document}